\documentclass{article}
\usepackage[utf8]{inputenc}
\usepackage[T2A]{fontenc}
\usepackage{amsmath,amssymb,amsthm,eucal,mathrsfs}
\usepackage[russian,english]{babel}
\usepackage{hyperref}

\usepackage{todonotes}

\newtheorem{proposition}{Proposition}
\newtheorem{theorem}{Theorem}
\newtheorem{lemma}{Lemma}

\newtheorem{corollary}{Corollary}

\theoremstyle{definition}
\newtheorem*{definition*}{Definition}

\newtheorem{assumption}{Assumption}

\theoremstyle{remark}
\newtheorem*{remark*}{Remark}

\newcommand \Conf {{\mathrm {Conf}}}

\newcommand{\SIN}{{\mathscr S}}
\newcommand{\PFSIN}{{\mathbf{K}_4}}
\newcommand{\Ai}{{\mathrm{Ai}}}
\newcommand{\AIRY}{{\mathscr A}}

\renewcommand \Prob {{\mathbb P}}
\DeclareMathOperator{\Pf}{Pf}

\begin{document}

\title{Sub-Poissonian estimates for exponential moments of additive functionals over pairs of~particles 
with respect to determinantal and symplectic Pfaffian point processes governed by~entire functions}

\author{Alexander I. Bufetov}

\AtEndDocument{\bigskip{\textsc{Steklov Mathematical Institute of Russian Academy of Sciences, Moscow, Russia,\\
		Department of Mathematics and Computer Sciences, St. Petersburg
		State University, St. Petersburg, Russia,\\
		Institute of Information Transmission Problems of Russian Academy of Sciences, Moscow, Russia,\\
	  Centre national de la recherche scientifique, France.}}}
	  \date{}

\maketitle

\begin{abstract}
	The aim of this note is to estimate the tail of the distribution of the number of particles in an interval under determinantal and Pfaffian point processes. The main result of the note is that the square of the number of particles under the determinantal point process whose correlation kernel is an entire function of finite order has sub-Poissonian tails. The same result also holds in the symplectic Pfaffian case. As a corollary, sub-Poissonian estimates are also obtained for exponential moments of additive functionals over pairs of particles.
\end{abstract}

\begin{quote}
	\selectlanguage{russian}
	\itshape
	Юлий Сергеевич Ильяшенко вошёл в мою жизнь чуть больше 30 лет назад, в марте 1993-го. Юлий Сергеевич читал лекцию о подкове Смейла, и преобразил мою жизнь, и стал моим учителем. Сегодня всё так же волнуюсь, рассказывая на его знаменитом семинаре.
\end{quote}
\selectlanguage{english}

\section{Introduction and the formulation of the main result}

Let $E$ be a Polish space, and let $\Conf(E)$ be the space of configurations on~$E$. Let $q\colon E\times E\to\mathbb{C}$ be a bounded compactly supported Borel function of two variables assuming value $0$ on the diagonal: $q(x,x)=0$. The additive functional~$\mathbf{S}_q$ over pairs of particles is a Borel function on the space $\Conf(E)$ defined by the formula
\begin{equation}\label{eq:Sq-def}
	\mathbf{S}_q(X)=\sum_{x\in X,y\in X}q(x,y).
\end{equation}
If $q$ is compactly supported then the sum in the right-hand side of~\eqref{eq:Sq-def} is finite.

The aim of this paper is to obtain sub-Poissonian estimates for exponential moments of additive functionals over pairs of particles with respect to determinantal and Pfaffian point processes on $\mathbb{R}$ whose correlation kernels are restrictions onto $\mathbb{R}$ of entire functions of finite order. 

If $q(x,y)=0$ once either $x$ or $y$ lies outside a compact interval $I$ then denoting by $\#_I(X)$ the number of particles of the configuration $X$ in the interval~$I$, of course, we have
\begin{equation*}
	|S_q(X)|\le \sup |q(x,y)|\cdot \#_I^2(X).
\end{equation*}
We therefore start with estimates for the exponential moment of the square of number of particles in an interval.

Determinantal point processes arising in random matrix theory are governed by entire functions of finite order, such as the exponential function, the Bessel function or the Airy function.
Starting with an elementary estimate of the determinant --- or the Pfaffian --- in terms of the divided differences of the values of the kernel using standard estimates for the derivatives of an entire function and applying the Laplace transform we obtain, for our exponential moments, upper estimates of the form
\begin{equation}\label{eq:sub-Poisson}
	\mathbb{E}\exp(\lambda\#_I^2)\le \exp(c_1(\exp(\lambda\sigma)-1)),
\end{equation}
where, furthermore, the constant $c_1$ depends on the interval $I$, while the constant $\sigma$ is the order of the entire function governing the correlation kernel of our point process.

Recall that a random variable $\xi$ assuming non-negative integer values has Poisson distribution with parameter $\theta>0$ if for all $k\in\mathbb{N}\cup\{0\}$ we have
\begin{equation*}
	\Prob(\xi=k)=e^{-\theta}\cdot\frac{\theta^k}{k!}.
\end{equation*}
The exponential moments of a Poisson random variable are given by the formula
\begin{equation*}
	\mathbb{E}e^{\lambda\xi}=e^{\theta(e^\lambda-1)}.
\end{equation*}
The main result of this note states, therefore, that the exponential moments of $\#_I^2$ do not grow faster that those of a multiple of a Poisson random variable. 

The key point in the argument is an estimate of the form
\begin{equation*}
	\Prob(\#_I\ge n)\le\exp\biggl(Bn^2-\frac{1}{2\sigma}n^2\log n\biggr),
\end{equation*}
where $B$ depends on $I$ and $\sigma$ is as above. For the $\mathrm{Sine}_{\beta}$ process for general $\beta$,  precise estimates for $\Prob(\#_I\ge n)$ are due to Holcomb and Valk\'o \cite{holcomb}; for the determinantal point process corresponding to the classical Fock space, precise estimate the tail probabilities for the number of particles in a ball have been earlier obtained by Manjunath Krishnapur \cite{krishnapur}, whose argument relies on the radial symmetry of the point process. The methods in this note are completely different both from those of Holcomb and Valk\'o \cite{holcomb}, and from those of Krishnapur \cite{krishnapur}.  

We proceed to the precise formulation.
Let $\Pi$ be a self-adjoint Hermitian projection kernel on $\mathbb{R}$. We keep the same symbol $\Pi$ for the corresponding locally trace class projection operator, and we let $\Prob_\Pi$ be the corresponding determinantal point process on $\mathbb{R}$ (for a general introduction to determinantal point processes see \cite{HoughEtAl}; in this note we follow the notation and conventions of~\cite{Buf-AnnProb}).

Let $I\subset\mathbb{R}$ be a compact interval. Impose the following assumption on our self-adjoint kernel $\Pi(x,y)$.

\begin{assumption}\label{asm:1}
	Let $z,w\in\mathbb{C}$ and let $\tilde\Pi(z,w)$ be a function of two variables that is entire in $z$ for any fixed $w$, satisfies $\tilde\Pi(z,w)=\overline{\tilde\Pi(w,z)}$ and assumes real values as soon as $z,w$ range over $\mathbb{R}$. We assume that there exist constants $A>0$, $M>0$, $\sigma>0$ and a positive continuous function $\rho$ on $I$ such that for any $p\in I$ we have the estimate
	\begin{equation*}
		|\tilde{\Pi}(p,z)|\le Ae^{M|z-p|^\sigma}\quad\text{for all}\quad z\in\mathbb{C},
	\end{equation*}
	and for any $x,y\in I$ we have
	\begin{equation*}
		\Pi(x,y)=\rho(x)\rho(y)\tilde{\Pi}(x,y).
	\end{equation*}
\end{assumption}

\begin{theorem}\label{thm:main}
	There exist positive constants $B$ and $c$ depending only on $A$, $M$, $\sigma$, $|I|$, and $\sup_I\rho$ such that under Assumption~\ref{asm:1} the number of particles $\#_I$ in the interval $I$ satisfies, under the determinantal point process $\Prob_\Pi$, the estimates
	\begin{equation}\label{eq:prob-bound}
		\Prob_{\Pi}(\#_I\ge n)\le\exp\biggl(Bn^2-\frac{1}{2\sigma}n^2\log n\biggr)
	\end{equation}
	and
	\begin{equation}\label{eq:thm-main}
		\mathbb{E}_{\Prob_\Pi}\exp(\lambda\#_I^2)\le
		\exp(c(\exp(\lambda\sigma)-1)).
	\end{equation}
\end{theorem}

We give an informal explanation for the super-normal decay of the probability $\Prob_{\Pi}(\#_I\ge n)$.
Observe that for the circular unitary ensemble 
\begin{equation*}
	\prod_{1\le k<l\le n} \bigl|e^{i\theta_k}-e^{i\theta_l}\bigr|^2\cdot
	\prod_{l=1}^n \frac{d\theta_l}{2\pi}
\end{equation*}
the probability that an interval of length $\varepsilon$ contains at least $r$ particles is of the order of magnitude $\varepsilon^{r(r-1)/2}$.

In order to see informally whence emerges the extra logarithm in the sine-process, the \emph{scaling} limit of circular unitary ensemble, consider the following analogy.

Let $\Prob_n$ be the scaled Legendre ensemble, that is, the orthogonal polynomial ensemble with uniform weight on the interval $[-n/2,n/2]$:
\begin{gather*}
	d\Prob_n(t_1,\dots,t_n)=
	\frac{1}{Z_n}\prod_{1\le i<j\le n}|t_i-t_j|^2 \cdot
	\prod_{i=1}^n dt_i,\\
	Z_n=n^{n(n-1)/2}\cdot 2^{n(n+1)/2}\cdot\prod_{k=0}^{n-1}\frac{(k!)^2}{(2k+1)!}
\end{gather*}
One verifies without effort the existence a positive constant $B>0$ such that for all $n\in\mathbb{N}$ we have
\begin{equation*}
	\Prob_n(\#_{[0,1]}\ge n)\le 
	\exp\biggl(Bn^2-\frac{n^2\log n}{2}\biggr).
\end{equation*}
Indeed, the factor $n^{-n(n-1)/2}$ is simply due to scaling, and all the remaining factors are subsumed into $B$.

The formula~\eqref{eq:prob-bound} is obtained
by comparing the determinant of the correlation kernel with the product of divided differences of increasing order and using standard estimates on the decay of Taylor---Maclaurin coefficients of an entire function. A corollary of sub-Poissonian bounds on the tail of the distribution of the square of the number of particles is an upper bound on the exponential moment of additive functionals over pairs of particles --- indeed, such an estimate follows directly using the negative associations of a determinantal point process governed by a self-adjoint contraction.

In order to obtain exponential estimates for additive functionals over pairs of particles, we need a uniform version of Assumption~\ref{asm:1}.
Let $U\subset\mathbb{R}$ be an open subset. 
\begin{assumption}\label{asm:2}
	Let $z,w\in\mathbb{C}$ and let $\tilde\Pi(z,w)$ be a function of two variables that is entire in $z$ for any fixed $w$, satisfies $\tilde\Pi(z,w)=\overline{\tilde\Pi(w,z)}$ and assumes real values as soon as $z,w$ range over $\mathbb{R}$. We assume that there exist constants $A>0$, $M>0$, $\sigma>0$ and a positive continuous function $\rho$ on $U$, bounded above, such that for any $p\in U$ we have the estimate
	\begin{equation*}
		|\tilde{\Pi}(p,z)|\le Ae^{M|z-p|^\sigma}\quad\text{for all}\quad z\in\mathbb{C},
	\end{equation*}
	and for any $x,y\in U$ we have
	\begin{equation*}
		\Pi(x,y)=\rho(x)\rho(y)\tilde{\Pi}(x,y).
	\end{equation*}
\end{assumption}

We are ready to formulate Corollary~\ref{cor:1}, an estimate for exponential moments of additive functionals over pairs of particles with possibly non-compact support. Let $q\colon\mathbb{R}^2\to\mathbb{C}$ be a function of two variables such that
\begin{equation*}
	\|q\|_{(1,\infty)}=
	\sum_{k,l\in\mathbb{Z}}\max_{\substack{|x-k|\le 1\\|y-l|\le 1}}
	|q(x,y)|<+\infty.
\end{equation*}
Consider the additive functional
\begin{equation*}
	\mathbf{S}_q(X)=\sum_{x\in X, y\in X}q(x,y)
\end{equation*}
Theorem~\ref{thm:main} immediately implies

\begin{corollary}\label{cor:1}
	Under Assumption~\ref{asm:2} there exists a positive constant $c$ depending only on $A$, $M$, $\sigma$, $|I|$, and $\sup_I\rho$ such that for all $\lambda>0$ we have
	\begin{equation*}
		\mathbb{E}_{\Prob_\Pi}\exp(\lambda \mathbf{S}_q)\le
		\exp\bigl(c\bigl(\exp(\|q\|_{1,\infty}\lambda\sigma)-1\bigr)\bigr).
	\end{equation*}
\end{corollary}

\begin{proof}[Proof of Corollary~\ref{cor:1}.] 
	Recall that an increasing function of a configuration of particles is by definition a function whose value does not decrease if a new particle is added to the configuration.

	Recall furthermore that a point process $\Prob$ is said to have negative associations if for any $k\in\mathbb{N}$, any precompact disjoint Borel subsets $C_1,\dots,C_k$ of the phase space and any bounded increasing Borel functions depending, respectively, only on the restriction of our configuration onto $C_1,\dots,C_k$, we have
	\begin{equation*}
		\mathbb{E}_\Prob f_1\cdots f_k\le \mathbb{E}_\Prob f_1 \cdots \mathbb{E}_\Prob f_k. 
	\end{equation*}

	Let $q\colon \mathbb{R}^2\to\mathbb{C}$ be a function of two variables satisfying $\|q\|_{(1,\infty)}<+\infty$.
	Let, furthermore. $\Phi(\lambda)$ be an increasing continuous function on $\mathbb{R}_+$ such that $\Phi(0)=1$ and that $\log\Phi(\lambda)$ is convex on $\mathbb{R}_+$. Such function is obviously super-multiplicative:
	\begin{equation*}
		\prod_{k=1}^n \Phi(\lambda_k)\le\Phi\biggl(\sum_{k=1}^n\lambda_k\biggr).
	\end{equation*}
	We then have the following

	\begin{proposition}\label{prop:est-Phi}
		Let $\Prob$ be a point process defined on an open subset $U\subset\mathbb{R}$ and having negative correlations.
		If for any $k\in\mathbb{Z}$ we have
		\begin{equation}\label{eq:expmoment-Phi}
		\mathbb{E} e^{\lambda\#_{[k,k+1]}^2}\le\Phi(\lambda),
		\end{equation}
		then for any $q$ satisfying $\|q\|_{(1,\infty)}<+\infty$ we also have
		\begin{equation*}
			\mathbb{E} e^{\lambda\mathbf{S}_q}\le\Phi(\lambda\|q\|_{(1,\infty)}).
		\end{equation*}
	\end{proposition}

\begin{proof} Indeed, we have
\begin{multline*}
	\mathbf{S}_q(X)\le 
	\sum_{k,l\in\mathbb{Z}}
	\max_{\substack{|x-k|\le 1\\|y-l|\le 1}}
	|q(x,y)|\cdot \#_{[k,k+1]}(X)\cdot \#_{[l,l+1]}(X)\le{}\\
	{}\le \sum_{k\in\mathbb{Z}}\#_{[k,k+1]}(X)^2
	\sum_{l\in\mathbb Z} \max_{\substack{|x-k|\le 1\\|y-l|\le 1}} |q(x,y)|.
\end{multline*}
\end{proof}

Corollary~\ref{cor:1} immediately follows from Proposition~\ref{prop:est-Phi}, where \eqref{eq:expmoment-Phi} holds with $\Phi(\lambda)=\exp(c(e^\lambda-1))$ by Theorem~\ref{thm:main}.
\end{proof}

The rest of the note is organized as follows. In Sections~\ref{sec:div-diff}--\ref{sec:end-proof} we prove Theorem~\ref{thm:main}: we start by estimating the determinant using divided differences (Section~\ref{sec:div-diff}), proceed by recalling standard estimates on the derivatives of an entire function (Section~\ref{sec:entire}) and conclude by taking the Laplace transform (Section~\ref{sec:end-proof}). We next apply Theorem~\ref{thm:main} to specific examples, the sine-kernel, the Bessel kernel, the Airy kernel, and projection kernels onto generalized Fock spaces.

\paragraph{Acknowledgements.} I am deeply grateful to Charles Bordenave and Alexey Klimenko for useful discussions and to the anonymous referee for constructive and helpful criticism. This research was supported by 
the grant of the Government of the Russian Federation for the state support of scientific research, carried out under the supervision of leading scientists (agreement 075-15-2021-602), 
and by the ANR REPKA (ANR-18-CE40-0035).

\section{Beginning of the proof of Theorem~\ref{thm:main}: divided differences and estimates for determinants}
\label{sec:div-diff}

This section contains the key simple estimate of the determinant using the divided differences.

For a determinantal point process on an open subset $U\subset\mathbb{R}$ governed by the kernel $\Pi$ and for a compact interval $I\subset U$, for the number of particles $\#_I$ in the interval $I$ by definition we have
\begin{multline}\label{eq:prob-det}
	\Prob_\Pi(\#_I\ge n)\le\mathbb{E}_{\Prob_\Pi} \#_I\cdot (\#_I-1)\cdots (\#_I-n+1)={}\\
	{}=\frac{1}{n!}\int_I\!\cdots\!\int_I \det\Pi(x_i,x_j)_{i,j=1,\dots,n}\cdot\prod_{i=1}^n dx_i.
\end{multline}
Our next aim is therefore to estimate the integral in the right-hand side.

Let $I\subset\mathbb{R}$ be an interval and $\Pi(x,y)$ be a function of two variables that is infinitely smooth in a neighbourhood of~$I$. Denote
\begin{equation}\label{eq:max-l}
	m_l(\Pi,I)=\frac{1}{l!}\max_{x,y\in I}\biggl|\frac{\partial^l \Pi(x,y)}{\partial y^l}\biggr|.
\end{equation}

\begin{proposition}
	For all $n\in\mathbb{N}$ the following inequality holds
	\begin{equation*}
		\biggl|\idotsint_{I^n}\det \Pi(x_i,x_j)_{i,j=1,\dots,n}
		\cdot\prod_{i=1}^n dx_i\biggr|\le 
		|I|^{n(n+1)/2}\cdot n! \cdot \prod_{l=0}^{n-1}m_l(\Pi,I).
	\end{equation*}
\end{proposition}

\begin{proof}
	Fix $x_1, \dots, x_n$. Introduce the divided differences 
	$$
		\Pi[x; x_1, x_2, \dots, x_l]
	$$
	inductively by writing 
	$$
		\Pi[x; x_1]=\Pi(x,x_1)
	$$
	and 
	$$
		\Pi[x; x_1, x_2, \dots, x_{l+1}]=\frac{\Pi[x; x_1, x_2, \dots, x_l]-\Pi[x; x_2, \dots, x_{l+1}]}{x_1-x_{l+1}}.
	$$
	By definition of the determinant, we have 
	$$
		\det(\Pi(x_i, x_j))_{i,j=1, \dots, n}=\Delta(x_1, \dots, x_n)\det(Q_{il}),
	$$
	where $$\Delta(x_1, \dots, x_n)=\prod_{1\leq i<j\leq n} (x_i-x_j)$$ is the Vandermonde 
	determinant of $x_1, \dots, x_n$ and we set
	$$
		Q_{il}({\vec x})=\Pi[x_i;  x_1, x_2, \dots, x_l].
	$$ 
	For any $x\in I$ there exists $y\in I$ such that 
	$$
		\Pi[x;x_1, \dots, x_l]=\frac1{(l-1)!}\Pi^{(l-1)}(x,y).
	$$
	Absolute value of each of $n!$ terms in $\det(Q_{il})$ can therefore be bounded by the product $\prod_{l=0}^{n-1}m_l(\Pi,I)$, while $|\Delta(x_1, \dots, x_k)|\le |I|^{n(n-1)/2}$, and the desired inequality is proved.
\end{proof}

A similar proposition holds for matrix-valued functions. Let $I\subset\mathbb{R}$, $r>0$, $K(x,y)$ be a infinitely smooth function in a neighbourhood of $I$ that takes values in the space of $r\times r$ matrices. Denote
\begin{equation}\label{eq:max-l-matrix}
	m_l(K,I)=\frac{1}{l!}\max_{\substack{x,y\in I\\ i,j=1,\dots,r}}\biggl|\frac{\partial^l K(x,y)_{ij}}{\partial y^l}\biggr|.
\end{equation}

\begin{proposition}
	For all $n\in\mathbb{N}$ the following inequality holds
	\begin{equation*}
		\biggl|\idotsint_{I^n}\det K(x_i,x_j)_{i,j=1,\dots,n}
		\cdot\prod_{i=1}^n dx_i\biggr|\le 
		|I|^{n+r\frac{n(n-1)}{2}}(rn)! \cdot \biggl(\prod_{l=0}^{n-1}m_l(K,I)\biggr)^r.
	\end{equation*}
\end{proposition}

\begin{proof}
	In the matrix $(K(x_i,x_j))_{i,j=1,\dots,n}$ we can make the same transformation as in the previous proof on each group of rows corresponding to the same row in $K(x,y)$ (that is, the rows $k,k+r,k+2r,\dots,k+(n-1)r$), hence we get
	$$
	\det(K(x_i, x_j))_{i,j=1, \dots, n}=\Delta(x_1, \dots, x_n)^r\det(Q_{il}),
	$$
	where elements in the first $r$ rows of the matrix $Q$ are bounded by $m_0(K,I)$, those in the next $r$ rows are bounded by $m_1(K,I)$, and so on. This gives the desired inequality.
\end{proof}

A similar proposition holds also for Pfaffians. Let $I\subset\mathbb{R}$, $r>0$, $\mathbf{K}(x,y)$ be a infinitely smooth function in a neighbourhood of $I$ that takes values in the space of $2\times 2$ matrices and satisfies the condition
\begin{equation*}
	\mathbf{K}(x,y)=-\mathbf{K}(y,x)^\top.	
\end{equation*}
As above, denote
\begin{equation*}
	\tilde m_l(\mathbf{K},I)=\max_{\substack{x,y\in I\\ i,j=1,2}}\biggl|\frac{\partial^l \mathbf{K}(x,y)_{ij}}{\partial y^l}\biggr|.
\end{equation*}

\begin{proposition}
	For all $n\in\mathbb{N}$ the following inequality holds
	\begin{equation*}
		\biggl|\idotsint_{I^n}\Pf\mathbf{K}(x_i,x_j)_{i,j=1,\dots,n}
		\cdot\prod_{i=1}^n dx_i\biggr|\le 
		|I|^{\frac{n(n+1)}{2}}\cdot \sqrt{(2n)!} \cdot \prod_{l=0}^{n-1}\tilde m_l(\mathbf{K},I).
	\end{equation*}
\end{proposition}

The proof immediately follows from the previous proposition, the Cauchy---Bunyakovsky---Schwarz inequality and the identity
\begin{equation*}
	(\Pf \mathbf{K}(x_i,x_j)_{i,j=1,\dots,n})^2=\det \mathbf{K}(x_i,x_j)_{i,j=1,\dots,n}.
\end{equation*}

\section{Estimates of entire functions}
\label{sec:entire}

We now estimate the maxima of derivatives~\eqref{eq:max-l} using the standard estimates (cf. Levin~\cite{Levin}) for the derivatives of an entire functions of finite order.

\begin{proposition}\label{prop:det-entire-est}
	If Assumption~\ref{asm:1} holds for the kernel~$\Pi$, then there exists $B>0$ that depends on $A$, $M$, $\sigma$, and $|I|$ only and such that for any $n\in\mathbb{N}$ and any $x_1,\dots,x_n\in I$ we have
	\begin{equation}\label{eq:det-est}
		\bigl|\det\Pi(x_i,x_j)_{i,j=1,\dots,n}\bigr|\le 
		\exp\biggl(Bn^2-\frac{1}{2\sigma}n^2\log n\biggr).
	\end{equation} 
\end{proposition}

This proposition has a direct Pfaffian analogue. Let $\mathbf{K}$ be a $2\times 2$-matrix-valued kernel
\begin{equation*}
	\mathbf{K}(x,y)=
	\begin{pmatrix}
		\mathbf{K}_{11}(x,y)&\mathbf{K}_{12}(x,y)\\
		\mathbf{K}_{21}(x,y)&\mathbf{K}_{22}(x,y)\\
	\end{pmatrix},
\end{equation*}
such that each of the functions $\mathbf{K}_{ij}$ satisfies the same uniform estimate
\begin{equation*}
	|\mathbf{K}_{ij}(p,z)|\le Ae^{M|z-p|^\sigma}.
\end{equation*}
Then there exists $B>0$ that depends on $A,M,\sigma$, and $|I|$ such that we have
\begin{equation}\label{eq:pf-est}
	|\Pf\mathbf{K}(x_i,x_j)|\le 
	\exp\biggl(Bn^2-\frac{1}{2\sigma}n^2\log n\biggr).
\end{equation}

For the proof of Proposition~\ref{prop:det-entire-est} and its Pfaffian analogue we now recall the following standard estimate (cf. Levin~\cite{Levin}).

\begin{lemma}
	Let $A>0$, $M>0$, $\sigma>0$, and $f(z)$ be an entire function such that
	\begin{equation*}
		|f(z)|\le Ae^{M|z|^\sigma}.
	\end{equation*}	
	Then for all $n$ we have
	\begin{equation*}
		\frac{1}{n!}\biggl|\frac{d^nf(0)}{dz^n}\biggr|\le Ae^{n+1}\biggl(\frac{n+1}{M}\biggr)^{-n/\sigma}.
	\end{equation*}
\end{lemma}

\begin{proof}
	One may apply the Cauchy formula
	\begin{equation*}
		\frac{1}{n!}\frac{d^nf(0)}{dz^n}=
		\frac{1}{2\pi i}\oint_{|z|=R}\frac{f(z)}{z^{n+1}}\,dz
	\end{equation*}
	with $R=\bigl(\frac{n+1}{M}\bigr)^{1/\sigma}$.
\end{proof}

\begin{remark*}
	Let $\Pi(x,y)=\rho(x)\rho(y)\tilde{\Pi}(x,y)$. Then
	\begin{equation*}
		\det\Pi(x_i,x_j)_{i,j=1,\dots,n}=\biggl(\prod_{i=1}^{n}\rho(x_i)\biggr)^2
		\cdot\det\tilde{\Pi}(x_i,x_j),
	\end{equation*}
	and we can apply the previous estimate to the function $\tilde{\Pi}$.
\end{remark*}

Substituting the above estimate~\eqref{eq:det-est} into~\eqref{eq:prob-det} we arrive at the estimate~\eqref{eq:prob-bound} in Theorem~\ref{thm:main}:
\begin{equation*}
	\Prob_{\Pi}(\#_I\ge n)\le\exp\biggl(Bn^2-\frac{1}{2\sigma}n^2\log n\biggr).
\end{equation*}

\section{Conclusion of the proof of Theorem~\ref{thm:main}}
\label{sec:end-proof}
	
Using the above estimate~\eqref{eq:prob-bound} for the probabilities of large values of $\#_I$, we now proceed to estimating the exponential moments. We start with a simple proposition.

\begin{proposition}\label{prop:int-est-tlogt}
	Let $B>0$, $\delta>0$, and $\rho$ be a nonnegative function on $[1,+\infty)$ satisfying 
	\begin{equation*}
		\rho(t)\le e^{Bt-\delta t\log t}.
	\end{equation*}
	Then there exist positive constants $c_1, c_2$ depending only on $B,\delta$ such that for all $\lambda>0$ the following inequality holds:
	\begin{equation*}
		\int_1^{+\infty}e^{\lambda t}\rho(t)\,dt\le 
		\exp\bigl(c_1\exp(\lambda/\delta)+c_2\bigr).
	\end{equation*} 
\end{proposition}

\begin{remark*}
	Precise explicit expressions can of course be given for $c_1,c_2$, but we do not need them.
\end{remark*}

\begin{proof}
	Denote
	\begin{equation*}
		S(t)=(\lambda+B)t-\delta t\log t
	\end{equation*}
	and write
	\begin{align*}
		S'(t)&{}=\lambda+B-\delta-\delta \log t,\\
		S''(t)&{}=-\delta/t.	
	\end{align*}
	The function $S$ is concave on the interval $(1,\infty)$ and reaches its maximum in the point
	\begin{equation*}
		t_0=\exp\biggl(\frac{\lambda+B-\delta}{\delta}\biggr).
	\end{equation*}
	Then
	\begin{equation*}
		\int_1^{5t_0/4} e^{S(t)}\,dt\le \frac{5t_0}{4} e^{S(t_0)}.
	\end{equation*}
	For $t>5t_0/4$ the following estimate holds:
	\begin{equation*}
		S(t)\le S(5t_0/4)+S'(5t_0/4)\biggl(t-\frac{5t_0}{4}\biggr)\le
		S(t_0)-\delta\log\frac{5}{4}\cdot\biggl(t-\frac{5t_0}{4}\biggr),
	\end{equation*}
	hence
	\begin{equation*}
		\int_{5t_0/4}^{+\infty} e^{S(t)}\,dt\le \frac{1}{\delta\log(5/4)} e^{S(t_0)}.
	\end{equation*}
	Combining these inequalities we obtain the desired estimate.
\end{proof}

Let us conclude the proof of Theorem~\ref{thm:main}.
The estimate~\eqref{eq:prob-bound} can be rewritten as
\begin{equation*}
	\Prob_{\Pi}(\#_I^2\ge t)\le\exp\biggl(\tilde B t-\frac{t}{\sigma}\log t\biggr),\quad t\ge 1.
\end{equation*}
Summation by parts yields
\begin{multline}\label{eq:est-expect}
	\mathbb{E}\exp(\lambda\#_I^2)=
	\sum_{k=0}^\infty e^{\lambda k}\bigl(\Prob(\#_I^2\ge k)-\Prob(\#_I^2\ge k+1)\bigr)={}\\
	{}=1+\sum_{j=1}^\infty \Prob(\#_I^2\ge j)\bigl(e^{\lambda j}-e^{\lambda(j-1)}\bigr)=
	1+\int_{1}^\infty \lambda e^{\lambda(t-1)}\Prob(\#_I^2\ge t)\,dt.
\end{multline}
By Proposition~\ref{prop:int-est-tlogt} we have
\begin{equation*}
	\mathbb{E}\exp(\lambda\#_I^2)\le 
	1+\lambda e^{-\lambda}\exp(c_1e^{\lambda\sigma}+c_2).
\end{equation*}
To get the desired estimate~\eqref{eq:thm-main} we consider separately $\lambda\in [0,1]$ and $\lambda\in [1,\infty)$. In the former case we have
\begin{equation*}
	\mathbb{E}\exp(\lambda\#_I^2)\le 1+\lambda\cdot \exp(c_1e^{\sigma}+c_2),
\end{equation*}
while the right-hand side of~\eqref{eq:thm-main} is not less than $1+\lambda\cdot c\sigma$, so for large $c$ we have~\eqref{eq:thm-main} for all $\lambda\in[0,1]$.

Consider now $\lambda\ge 1$. Here
\begin{equation*}
	\mathbb{E}\exp(\lambda\#_I^2)\le 1+e^{-1}\exp(c_1e^{\lambda\sigma}+c_2)\le
	\exp(c_1e^{\lambda\sigma}+c_2+\log 2).	
\end{equation*}
It remains to observe that if $c$ is large enough we have 
\begin{equation*}
	c_1e^{\lambda\sigma}+c_2+\log 2\le c(e^{\lambda\sigma}-1)
\end{equation*}
for all $\lambda>1$.
Theorem~\ref{thm:main} is proved completely.

This final argument can be seen in the light of the following simple general remark.

\begin{lemma}Let $\Psi$ be an increasing convex function, $\Psi(0)=1$, $\Psi'(0)>0$. Then there exists a constant $d>0$ such that for all $\lambda>0$ we have
	\begin{equation*}
		\min(e^{\Psi(\lambda)}, 1+\lambda e^{\Psi(\lambda)})\le e^{d(\Psi(\lambda)-1)}.
	\end{equation*}
\end{lemma}

\begin{proof}
	First note that for $\lambda\ge 1$ we have $e^{\Psi(\lambda)}\le 1+\lambda e^{\Psi(\lambda)}$, hence we need that
	\begin{equation*}
		\Psi(\lambda)\le d(\Psi(\lambda)-1).
	\end{equation*}
	The latter inequality holds for all $\lambda>1$ if $\Psi(1)\ge 1+1/(d-1)$, that is, for $d\ge \Psi(1)/(\Psi(1)-1)$.
	
	Consider $\lambda\in[0,1]$. We have
	\begin{equation*}
		\min(e^{\Psi(\lambda)}, 1+\lambda e^{\Psi(\lambda)})\le 
		1+\lambda e^{\Psi(\lambda)}\le 1+\lambda e^{\Psi(1)}.
	\end{equation*}
	Convexity yields $\Psi(\lambda)-1\ge \Psi'(0)\lambda$, 
	so if $d\ge e^{\Psi(1)}/\Psi'(0)$, we have the desired inequality for all $\lambda\in [0,1]$.
\end{proof}

\section{The case of the sine-process}

For general $\beta$, Holcomb and Valk\'o \cite{holcomb}  obtained precise estimates for the probabilities 
$\Prob(\#_I\ge n)$ under the $\mathrm{Sine}_{\beta}$-process. To the best of my understanding, the validity of the estimate \eqref{eq:main-sin-2}, whose proof, at least, in the scheme used in this paper, relies on the negative correlations for determinantal point processes, remains open for general $\beta$, indeed, already for $\beta=4$.

The sine-process, the scaling limit of radial part of Haar measures on unitary groups of growing dimension is a point process $\Prob_{\SIN}$ on $\mathbb{R}$, whose correlation measures have the form
\begin{equation*}
	\rho_n(x_1,\dots,x_n)\,dx_1\dots dx_n,
\end{equation*}
where $\rho_n$, the $n$-th correlation function, is given by the formula
\begin{equation*}
	\rho_n(x_1,\dots,x_n)=\det\biggl(\frac{\sin\pi(x_i-x_j)}{\pi(x_i-x_j)}\biggr)_{i,j=1,\dots,n}.
\end{equation*}

Let us state the main result of the note in the case of the sine-process. 
\begin{proposition}\label{prop:main-sin}
	Let $I\subset\mathbb{R}$ be a interval of length~1. 
	Then there exists a constant $c_1>0$ such that for all $\lambda>0$ we have
	\begin{equation}\label{eq:main-sin-1}
		\mathbb{E}_{\Prob_{\SIN}}e^{\lambda\#_I^2}\le \exp\bigl(c_1\bigl(\exp(\lambda)-1\bigr)\bigr).
	\end{equation}
	Similarly, for any $q$ satisfying $\|q\|_{(1,\infty)}<+\infty$ we have
	\begin{equation}\label{eq:main-sin-2}
		\mathbb{E} e^{\lambda\mathbf{S}_q}\le 
		\exp\bigl(c_1\bigl(\exp(\lambda\|q\|_{(1,\infty)})-1\bigr)\bigr).
	\end{equation}
\end{proposition}

Denote
\begin{equation*}
	S(t)=\frac{\sin\pi t}{\pi t}
\end{equation*} 
and introduce a matrix-valued kernel
\begin{equation*}
	\mathbf{K}_4(x,y)=\frac{1}{2}
	\begin{pmatrix}
		-IS(x-y)& S(x-y)\\
		-S(x-y)&DS(x-y)
	\end{pmatrix},
\end{equation*}
where
\begin{equation*}
	DS(t)=\frac{dS}{dt}(t),\quad
	IS(t)=\int_0^t S(u)\,du.
\end{equation*}
By definition the matrix $\mathbf{K}_4$ is antisymmetric for any $x,y\in\mathbb{R}$.

The symplectic sine-process is a point process $\Prob_{\PFSIN}$ on $\mathbb{R}$, whose correlation functions have the form
\begin{equation*}
	\rho_n(x_1,\dots,x_n)=\Pf(\mathbf{K}_4(x_i,x_j))_{i,j=1,\dots,n}.
\end{equation*}

The symplectic sine-process arises, for instance, as the scaling limit for the radial parts of  quaternionic-Hermitian matrices of growing size.  

Let us state the main result of the note in the case of the Pfaffian symplectic sine-process. 
\begin{theorem}\label{thm:main-Pfsin}
	Let $I\subset\mathbb{R}$ be a interval of length~1. 
	Then there exists a constant $c_1>0$ such that for all $\lambda>0$ we have
	\begin{equation*}
		\mathbb{E}_{\Prob_{\PFSIN}}e^{\lambda\#_I^2}\le
		\exp\bigl(c_1\bigl(\exp(\lambda)-1\bigr)\bigr).
	\end{equation*}
\end{theorem}

\begin{remark*}
	In the specific case of the sine-process the estimate of this note can be made more precise using sharp bounds obtained by Bonami---Jaming---Karoui~\cite{BoJaKa} on the eigenvalue of the sine-kernel restricted to an interval and recalling the theorem of Hough---Krishnapur---Peres---Vir\'ag~\cite{HoughEtAl} to the effect that the distribution under a determinantal point process of the number of particles in a compact subset is that of the infinite sum of Bernoulli random variables whose probabilities of success are precisely the eigenvalues of the kernel, restricted onto the compact set.
	
	The argument of Bonami---Jaming---Karoui relies on the existence of a second-order differential operator commuting with the sine-kernel restricted onto an interval. The existence of such an operator is only established for a few specific examples: the sine, the Bessel, and the Airy kernels. Furthermore, to the best of my knowledge, extending the sharp bounds of Bonami---Jaming---Karoui onto the case of the Bessel and the Airy kernels remains an open problem. 
\end{remark*}

\section{The case of the Bessel kernel}

Let $s>-1$. Recall that the Bessel kernel $J_s$ with parameter $s$ is given by the formula
\begin{equation*}
	J_s(x,y)=\frac{\sqrt{x}J_{s+1}(\sqrt{x})J_{s}(\sqrt{y})-
	\sqrt{y}J_{s+1}(\sqrt{y})J_{s}(\sqrt{x})}{2(x-y)}.
\end{equation*}
The Bessel kernel induces a determinantal point process $\Prob_{J_s}$ of $\Conf((0,+\infty))$.

Note that the Bessel function $J_s(\sqrt{x})$ has the form
\begin{equation*}
	J_s(\sqrt{x})=\biggl(\frac{x}{4}\biggr)^{s/2}\sum_{m=0}^{\infty}
	\frac{(-1)^m}{m!\Gamma(m+s+1)}\cdot\biggl(\frac{x}{4}\biggr)^m.
\end{equation*}
The function $(x/4)^{-s/2}J_s(\sqrt{x})$ is entire of order~1. Thus setting $\rho(x)=(x/4)^{s/2}$ we note that the Bessel kernel satisfies Assumption~\ref{asm:1} with $\sigma=1$. 

The analogues of Proposition~\ref{prop:main-sin} will be slightly different for different values of $s$.

\begin{proposition}\label{prop:main-Bessel}
	Let $s\ge0$ and let $I=[a,b]$ be a subinterval of the closed half-line $[0,+\infty)$. Then there exists a constant $c_3>0$ such that for all $\lambda>0$ we have
	\begin{equation}\label{eq:main-Bessel-1}
		\mathbb{E}_{\Prob_{J_s}}\exp(\lambda\#_I^2)\le
		\exp(c_3(\exp(\lambda)-1)).
	\end{equation}
	If $s\in(-1,0)$, then the estimate~\eqref{eq:main-Bessel-1} holds for the subintervals $I$ of the open half-line $(0,+\infty)$.
\end{proposition}

For additive functionals over pairs of particles, the estimate~\eqref{eq:main-sin-2} of Proposition~\ref{prop:main-sin} only follows from Proposition~\ref{prop:main-Bessel} if $s=0$ or if $s\in(-1,0)$ and the support of the additive functional $q$ lies in $[a,+\infty)^2$, $a>0$. It would be very interesting to extend these estimates also to the remaining cases.

\section{The case of the Airy kernel}

We next consider the case of the Airy kernel and of its symplectic Pfaffian analogue. Let $\Ai(x)$ be the Airy function
\begin{equation*}
	\Ai(x)=\frac{1}{\pi}\int_0^{+\infty}\cos\biggl(\frac{t^3}{3}+xt\biggr)\,dt.
\end{equation*}
The Airy kernel is given by the formula
\begin{equation*}
	\AIRY(x,y)=\frac{\Ai(x)\Ai'(y)-\Ai(y)\Ai'(x)}{x-y}.
\end{equation*}
The Airy kernel induces on $\mathbb{R}$ a determinantal point process $\Prob_{\AIRY}$, for which we also obtain estimates for the exponential moment of the square of the number of particles. Indeed, the Airy function satisfies the following estimate, that can be made uniform on any half-line of the form $(a,+\infty)$, $a\in\mathbb{R}$.

\begin{proposition}
	For any $a\in\mathbb{R}$ there exists constant $C>0$, $M>0$, depending only on $a$ and such that for all $p>a$ and $z\in\mathbb{C}$ we have
	\begin{equation*}
		|\Ai(z)|\le Ce^{M|z-p|^{3/2}}.
	\end{equation*}
\end{proposition}

From Theorem~\ref{thm:main} we now obtain the following corollary.
\begin{proposition}
	Let $a\in\mathbb{R}$, and let $I\subset (a,+\infty)$ be a compact interval of length~1. There exists a constant $d$, depending only on~$a$, such that for all $\lambda>0$ we have
	\begin{equation*}
		\mathbb{E}_{\Prob_\AIRY}\exp(\lambda\#_I^2)\le
		\exp\bigl(d\bigl(\exp(3\lambda/2)-1\bigr)\bigr).
	\end{equation*}
\end{proposition}

For additive functionals over pairs of particles we also obtain the corresponding estimate.

\begin{corollary}
	Let $q(x,y)$ be a function of two variables satisfying $\|q\|_{(1,+\infty)}<\infty$ and supported in a half-line of the form $(a,+\infty)$. We then have
	\begin{equation*}
		\mathbb{E}_{\Prob_\AIRY}\exp(\lambda\mathbf{S}_q)\le
		\exp\biggl(d\biggl(\exp\biggl(\frac{3}{2}\|q\|_{(1,\infty)}\lambda\biggl)-1\biggr)\biggr).
	\end{equation*}
\end{corollary}

An estimate for the exponential moment of the square of the number of particles in an interval is also obtained for the symplectic Pfaffian Airy process.
Following Tracy and Widom \cite{TracyWidom}, we set
\begin{gather*}
	\AIRY^{(4)}_{11}(x,y)=-\frac{1}{2}\int_x^{+\infty}\AIRY(u,y)\,du+
	\frac{1}{4}\int_x^{+\infty}\Ai(u)\,du\int_y^{+\infty}\Ai(v)\,dv,\\
	\AIRY^{(4)}_{22}(x,y)=\frac{1}{2}\partial_y\AIRY(x,y)+
	\frac{1}{4}\Ai(x)\Ai(y),\\
	\AIRY^{(4)}_{12}(x,y)=-\AIRY^{(4)}_{21}(y,x)=
	\frac{1}{2}\AIRY(x,y)-\frac{1}{4}\Ai(y)\int_x^{+\infty}\Ai(u)\,du.
\end{gather*}
Convenient contour integral representations for the matrix Airy kernel are given by Baik---Barraquand---Corwin---Suidan in \cite{BBCS}.
For the Pfaffian point process induce by the symplectic Airy matrix kernel we thus obtain the estimate
\begin{equation*}
	\mathbb{E}_{\Prob_{\AIRY^{(4)}}}\exp(\lambda\#_I^2)\le
	\exp\bigl(d\bigl(\exp(3\lambda/2)-1\bigr)\bigr).
\end{equation*}
the constant $d$ is again uniform for all intervals of fixed length lying in a fixed positive half-line.

\section{The case of the generalized Fock kernel}

We now pass from point processes on $\mathbb{R}$ to point processes on $\mathbb{C}$  whose correlation kernels are holomorphic in the first variable and anti-holomorphic in the second. More precisely, we allow the second variable to range in a compact subset of $\mathbb{C}$ and assume that the correlation kernel considered as a function of the first variable with the value of the second fixed, is an entire function of finite order. 

We consider the case of generalized Fock spaces  on $\mathbb{C}$ \cite{BufetovQiu, Christ}. For the classical Fock space and the corresponding Ginibre ensemble precise estimates for the tail probabilities of the number of particles in a ball are due to Manjunath Krishnapur \cite{krishnapur} who relied on the radial symmetry of the underlying process.
   
Let $\Delta$ be the Euclidean Laplacian. Let $\psi\colon\mathbb{C}\to\mathbb{R}$ be a $C^2$-smooth function satisfying
\begin{equation*}
	m\le \Delta\psi\le M.
\end{equation*}
Let $\mathcal{F}_\psi$ be the generalized Fock space of entire functions on $\mathbb{C}$, square integrable with weight
\begin{equation*}
	e^{-2\psi(z)}\,d\mathrm{Leb}(z).
\end{equation*}
We assume that the space $\mathcal{F}_\psi$ is closed in $L_2(\mathbb{C},e^{-2\psi(z)}\,d\mathrm{Leb}(z))$.

The generalized Fock space $\mathcal{F}_\psi$ admits a reproducing kernel that we denote $B_\psi(z,w)$. The orthogonal projection $\Pi^\psi$ from $L_2(\mathbb{C},e^{-2\psi(z)}\,d\mathrm{Leb}(z))$ onto the closed subspace $\mathcal{F}_\psi$ takes the form
\begin{equation*}
	\Pi^\psi f(z)=
	\int_{\mathbb{C}}f(w)B_\psi(z,w)e^{-2\psi(w)}\,d\mathrm{Leb}(w).
\end{equation*}
Let $\Prob_{\psi}$ be the determinantal point process induced by the projection operator $\Pi^\psi$.

As before, we let $K\subset\mathbb{C}$ be a compact subset.

\begin{proposition}
	If $\sigma\ge 1$, and the weight $\psi$ satisfies the equality
	\begin{equation*}
		|\psi(z)|\le M|z|^\sigma,
	\end{equation*}
	then there exists $d>0$, depending on $K$ such that
	\begin{equation*}
		\mathbb{E}_{\Prob_{\psi}}\exp(\lambda\#_K^2)\le
		\exp\bigl(d\bigl(\exp(\lambda\sigma)-1\bigr)\bigr).
	\end{equation*}
\end{proposition}

\begin{proof}
	The proof directly follows from the Christ estimate for the generalized Fock kernel:
	\begin{equation*}
		|B_\psi(z,w)|^2e^{-2\psi(z)-2\psi(w)}\le\exp(-\alpha|z-w|).
	\end{equation*}
\end{proof}

For the classical Fock space, the weight $\psi$ is given by the formula $\psi(z)=|z|^2$ and the corresponding kernel $B_2(z,w)$ is explicitly given by the formula $B_2(z,w)=e^{z\bar{w}}$. Let $\Prob_{\mathrm{Fock}}$ be the corresponding determinantal point process. Since for fixed $w$ the function $e^{z\bar{w}}$ is of order $1$ in $z$, in the classical case we have a more precise estimate.

\begin{proposition}
	For any compact subset $K\subset\mathbb{C}$ there exists $d>0$ such that
	\begin{equation*}
		\mathbb{E}_{\Prob_{\mathrm{Fock}}}\exp(\lambda\#_K^2)\le
		\exp(d(\exp(\lambda)-1)).
	\end{equation*}
\end{proposition}

In this particular case, a different proof of the proposition can be given using the radial symmetry of our process, cf. Krishnapur \cite{krishnapur}. Observe, however, that for generalized Fock spaces the radial symmetry does not hold.

Using the stationarity of the classial Fock determinantal point process, we directly obtain an estimate also for multiplicative moments of additive functionals over pairs of particles.

For a function $q(x,y)$ of two complex variables, write
\begin{equation*}
	\|q\|_{(1,\infty)}=\sum_{k,l\in\mathbb{Z}+i\mathbb{Z}}\max_{\substack{|x-k|\le 1\\|y-l|\le 1}}|q(x,y)|.
\end{equation*}
From the upper bound on the expectation $\mathbb{E}_{\Prob_{\mathrm{Fock}}}\exp(\lambda\#_K^2)$, we derive, as above, the upper bound on $\mathbb{E}_{\Prob_{\mathrm{Fock}}}\exp(\lambda\mathbf{S}_q)$.
\begin{corollary}
	If $\|q\|_{(1,\infty)}<+\infty$, then for all $\lambda>0$ we have
	\begin{equation*}
		\mathbb{E}_{\Prob_{\mathrm{Fock}}}\exp(\lambda\mathbf{S}_q)\le
		\exp(d(\exp(\lambda)-1)).
	\end{equation*}
\end{corollary}

\section{Concluding remarks}

For the symplectic analogue of the Bessel kernel it does not seem clear how one should proceed: the kernel is not entire, and the square roots appearing in the definition of the Pfaffian Bessel kernel do not factor out as in the determinantal case. The situation is even worse when one considers real Pfaffian analogues of our point processes: here, even in the case of the sine-process, these kernel has a discontinuity on the diagonal, and an adaptation of the method is needed.

The proof of the estimates on exponential moments of additive functionals over pairs of particles relies on the negative associations, or, in other words, the repulsion of particles under a determinantal point process. Pfaffian point processes do not enjoy the negative association property, and the argument of this note does not apply directly to Pfaffian point processes.

It is tempting to conjecture, however, that negative associations used in this note could be replaced by estimates on the decay of correlations and Theorem~\ref{thm:main-Pfsin} could thus be extended to Pfaffian point processes.

\end{document}